\documentclass[oneside,english]{amsart}
\usepackage[T1]{fontenc}
\usepackage[latin9]{inputenc}
\usepackage{geometry}
\usepackage{tipa}

\makeatletter
 \theoremstyle{plain}
 \newtheorem{thm}{Theorem}[section]
  \theoremstyle{definition}
  \newtheorem{defn}[thm]{Definition}
  \theoremstyle{remark}
  \newtheorem{rem}[thm]{Remark}
  \theoremstyle{plain}
  \newtheorem{prop}[thm]{Proposition}
  \theoremstyle{plain}
  \newtheorem{cor}[thm]{Corollary}

\newcommand\rest{\hbox{\raise.17em\hbox{$ |\kern-.2em$}\lower.23em\hbox{$-$}}}
\newcommand\avint{\hbox{\hbox{$\displaystyle \int$}\hbox{\kern-.9em{$-$}}}}
\newcommand\smavint{\hbox{\hbox{$\int$}\hbox{\kern-.75em{$-$}}}}
\newcommand{\R}{\mathbb{R}}

\newcommand{\Z}{\mathbb{Z}}

\newcommand{\Pa}{\mathbb{P}}

\makeatother

\usepackage{babel}

\begin{document}

\title{Minimal surfaces in circle bundles over Riemann surfaces}

\author{Pablo M. Chac{\'o}n and David L. Johnson}

\curraddr{Pablo M. Chac{\'o}n \\
Departamento de Matem\'{a}ticas \\
Universidad de Salamanca\\
Plaza de la Merced, 1-4, 37008\\
Salamanca, Spain}

\curraddr{David L. Johnson\\
Department of Mathematics\\
Lehigh University\\
Bethlehem, Pennsylvania 18015-3174 USA}

\email{pmchacon@usal.es, david.johnson@lehigh.edu}

\thanks{The first author was partially supported by MEC project MTM2007-60017
and Fundaci\'{o}n S\'{e}neca project 04540/GERM/06, Spain. The second
author was partially supported by the Universitat de Val$\grave{\text{\textipa{e}}}$ncia
during his sabbatical stay in Spain.}

\subjclass[2000]{Primary: 53A10 Secondary 53C20}

\keywords{minimal surfaces, rectifiable sections, singularities}

\date{\today}

\begin{abstract}
For a compact 3-manifold $M$ which is a circle bundle over a Riemann
surface $\Sigma$ with even Euler number $e(M)$, and with a Riemannian
metric compatible with the bundle projection, there exists a compact
minimal surface $S$ in $M$. $S$ is embedded and is a section of
the restriction of the bundle to the complement of a finite number
of points in $\Sigma$. 
\end{abstract}
\maketitle

\section{Introduction}

Let $M$ be a 3-manifold which is a circle bundle over a compact Riemann
surface $\Sigma$, with projection $\pi:M\to\Sigma$. Assume that
the metric on $M$ is compatible with the bundle projection, that
is, $\pi$ is a Riemannian submersion and the fibers are geodesics.
Assume that the Euler class $e(M)$ of the associated rank-2 vector
bundle $E$ over $\Sigma$ is even. The goal of this paper is to show
the existence of a smooth minimal surface in $M$. We will show the
existence of such a surface which is a section of the bundle except
over a finite set of points, and is topologically the Riemann surface
$\Sigma$ with a finite number of cross-caps. 

An example of such a minimal surface is described in \cite{BG}. Consider
$M=T_{1}(S^{2})$, the unit tangent bundle of the standard 2-sphere.
For any choice of a unit tangent vector $v$ at $p\in S^{2}$ , the
Pontryagin cycle $P$, the section defined by parallel translation
of $v$ along each longitude line from $p$, will be a smooth minimal
surface in $M$ which is a smooth section except over $-p$. $P$
is in this case a totally-geodesic $\R\Pa^{2}$ embedded in $M$. 

The authors thank Antonio Ros for suggesting this problem, and also
Olga Gil Medrano for many helpful suggestions. The second-named author
thanks the Universidad de Granada and the Universidad de Salamanca
for their support during his visits.

\section{Minimal graphs}

Let $X$ be a compact, $n$-dimensional manifold, and let $\pi:B\to X$
be a fiber bundle over $X$ with with compact fiber $F$ of dimension
$k$. Any such bundle admits a class of Riemannian metrics, due to
Sasaki \cite{Sasaki}, for which the projection $\pi$ is a Riemannian
submersion with totally-geodesic fibers isometric to $F$ under inclusion,
determined by a choice of connection on the associated principal bundle. 

\begin{defn}
A \emph{rectifiable section} $T$ in $B$ is a countably-rectifiable,
integer-multiplicity $n$-current in $B$ so that, 
\begin{enumerate}
\item $\left\langle \overrightarrow{T}(q),\mathbf{e}(q)\right\rangle \geq0$,
$\left\Vert T\right\Vert $-almost everywhere; where $\mathbf{e}(q)$
is the unique horizontal (orthogonal to the fibers) $n$-plane at
$q$ which maps onto $T_{*}(X,\pi(q))$ under $\pi_{*}$ (preserving
orientation), and $\overrightarrow{T}$ is the unit oriented $n$-vector
tangent plane of $T$ at $q$.
\item The image current $\pi_{\#}(T)$ is the fundamental class $1[X]$
as an $n$-dimensional current on $X$ with integer coefficients.
\item If $\partial X=\emptyset$, $\partial T\equiv0(mod\,2)$ as flat chains
modulo 2 (If $\partial X\neq\emptyset$, $\partial T$ must have support
contained in $\pi^{-1}(\partial X)=\partial B$).
\end{enumerate}
The space of all such \emph{rectifiable sections} of the bundle $B$
over $X$ will be denoted $\widetilde{\Gamma}(B)$.

\end{defn}
\begin{rem}
This definition differs slightly from that in \cite{JS2}, in that
here the currents are only required to be relative cycles mod-2. This
is necessary because the currents constructed will not be cycles as
integral chains. As rectifiable currents they are by definition oriented,
but currents corresponding to smooth submanifolds may be nonorientable
as manifolds, or, equivalently, may have interior boundaries as rectifiable
currents. Compact nonorientable manifolds without boundary (as manifolds)
are mod-2 cycles as flat chains modulo 2. 
\end{rem}
In \cite{JS2} it is shown that any homology class of rectifiable
sections has a minimal-mass representative, which is a continuous
section over an open, dense subset of $X$. As remarked in \cite[11.1]{Morgan}
or \cite[4.2.26]{GMT}, the extension from integer coefficients to
$\Z/2\Z$ coefficients, and considering the currents as flat chains
modulo 2 for the boundary condition will not alter the arguments of
\cite{JS2}. 

If $\sigma$ is a $C^{1}$ section, then the mass of the image (usually
called the \emph{volume} of the section in this case) is given by
\[
\mathcal{V}(\sigma):=\int_{X}\sqrt{1+\left\Vert \nabla\sigma\right\Vert ^{2}+\cdots+\left\Vert \nabla\sigma\wedge{\min\left\{ n,k\right\} \atop \cdots}\wedge\nabla\sigma\right\Vert ^{2}}dV_{X}.\]

\begin{thm}
\textbf{\emph{\label{thm:JS}\cite{JS2}}} Let $X$ be a compact manifold,
and let $B$ be a fiber bundle over $X$ with compact smooth fiber
$F$ and with an associated Sasaki metric. In any nonempty mod-2 homology
class of rectifiable sections $\sigma:X\to B$, there is a mass-minimizing,
rectifiable section which is continuous except over a set $S$ of
measure 0 in $X$.
\end{thm}
\begin{rem}
It should be noted that the original result shows the section to be
$C^{1}$ on an open dense set; continuity may hold on a slightly larger
set. Also, the theorem does not say there will not be other mass-minimizers
that may have worse regularity, only that there is one which is this
nicely behaved. Finally, Proposition (\ref{pro:nonempty}) below will
imply that there is such a nonempty homology class of rectifiable
sections in the cases we need.
\end{rem}
Now, let $M$ be a 3-manifold which is a circle bundle $\pi:M\to\Sigma$
over a compact Riemann surface $\Sigma$, with a Sasaki metric. The
Euler class $e(M)$ of $M$ is the Euler class of the associated orientable
rank-2 vector bundle $E\to\Sigma$. In order to show the existence
of the claimed minimal surface in $M$, we first have to show that
$\widetilde{\Gamma}(M)\neq\emptyset$ when $e(M)$ is even.

\begin{prop}
\label{pro:nonempty}If $M\to\Sigma$ is a circle bundle over a compact
Riemann surface, with even Euler class $e(M)$, then $\widetilde{\Gamma}(M)\neq\emptyset$.
\end{prop}
\begin{proof}
If $k:=|e(M)|/2$, choose $k$ points $\{p_{1},\dots,p_{k}\}\subset\Sigma$.
Essentially by the Poincar\'{e}-Bendixon theorem, there is a smooth
section of $E$ with zeros only at the points $p_{j}$, of index $\pm2$
at all points, where the sign is that of $e(M)$. Equivalently, given
$\epsilon>0$ sufficiently small, there is a smooth section $\tau$
of $\left.M\right|_{\Sigma\backslash\{B_{\epsilon}(p_{1}),\dots,B_{\epsilon}(p_{k})\}}$
with the following boundary conditions: for each $j$, $\left.M\right|_{\partial B_{\epsilon}(p_{j})}\cong S^{1}\times S^{1}$,
so the map $z\mapsto z^{n}$ defines a section on the boundary component
$\left.M\right|_{\partial B_{\epsilon}(p_{j})}$ with index $n=\pm2$
for each $j\leq k$. A smooth section $\tau$ exists with these boundary
conditions, where we choose the sign of $n$ to match the sign of
$e(M)$. This section can be constructed to extend, for any $0<\delta<\epsilon$,
to $\left.M\right|_{\Sigma\backslash\{B_{\delta}(p_{1}),\dots,B_{\delta}(p_{k})\}}$
with similar boundary conditions. The limit of these extensions, as
$\delta\to0$, has closure which is a mod-2 cycle in $M$, as in \cite{BG}.
This limit is a rectifiable section, so the space $\widetilde{\Gamma}(M)$
of rectifiable sections is nonempty. 
\end{proof}
Following \cite{JS2}, with the slight modification to the boundary
conditions, there is a mass-minimizing rectifiable section $\sigma$
in the mod-2 homology class of $\tau$ above, which is a continuous
section over an open dense subset. We show below that the exceptional
set is a finite collection of fibers over points $\{x_{1},\dots,x_{n}\}\subset\Sigma$,
and that the mod-2 cycle which is the closure of this section is a
smooth minimal surface in $M$. We emphasize that the exceptional
points of the minimizer need not be the points used in Proposition
(\ref{pro:nonempty}); in particular, the number $n$ of points may
be larger than $k$.

\section{Singularities}

Consider now a mass-minimizing rectifiable section $T$ of a circle
bundle $M$ over a compact Riemann surface $\Sigma$, with the metric
as described earlier. An \emph{exceptional point, }or a\emph{ singular
point} $x\in\Sigma$ is a point over which $T$ is not a continuous
section. Since $\pi(Supp(T))=\Sigma$, this implies that there are
two points, at least, in $\pi^{-1}(x)\cap Supp(T)$ for an exceptional
point $x$. Our first goal will be to show that the entire fiber is
enclosed in $T$ over any exceptional point. This step uses a basic
construction which will be needed elsewhere as well, a horizontal
sequence of stretches of the current.

\subsection{H-cones}

If $S\in\widetilde{\Gamma}(M)$ is a rectifiable section with finite
mass, and if $x_{0}\in\Sigma$ is an arbitrary point, then for sufficiently
small $r>0$, and for all $\lambda>1$, $S$ defines a rectifiable
section $S_{\lambda,R}$ in $B(0,R)\times S^{1}$ by $S_{\lambda,R}=\left[\left(\phi_{\lambda}\right)_{\#}\left(S\rest\pi^{-1}(B(0,r))\right)\right]\rest B(0,R)\times S^{1}$,
if $\lambda r\geq R$, where $x_{0}$ corresponds with the center
$0$ of the coordinate system, $\phi_{\lambda}(x,y)=(\lambda x,y)$,
and $\pi^{-1}(B(0,R))$ is identified with $B(0,R)\times S^{1}$,
having the Riemannian metric induced from $M$ and the dilation $\phi_{\lambda}$.
$B(0,\lambda r)\times S^{1}$ also has a specific Riemannian metric,
the metric from $M$ stretched horizontally by $\phi_{\lambda}$.
Clearly, for an arbitrary $R>0$, if $\lambda>1$ is sufficiently
large, $S_{\lambda,R}$ will be well-defined in $\widetilde{\Gamma}(B(0,R)\times S^{1})$.

An \emph{h-cone} $H$ of $S$ at $x_{0}\in\Sigma$, for a given sequence
$\lambda_{i}\to\infty$, is the limit, for each $R>0$, of the sequence
of restricted stretches $S_{\lambda_{i},R}$, if that limit exists
as a rectifiable section (thinking of $S_{\lambda,R}$ as rectifiable
sections of $B(0,R)\times S^{1}$ in order to define the limit), \emph{and}
if $H_{\lambda,R}=H\rest B(0,R)\times S^{1}$ for all $\lambda>1$.
The limit will be a rectifiable section of $B(0,R)\times S^{1}$ with
the flat Euclidean metric.

For a given point, current, and sequence of stretches, an h-cone may
or may not exist, just as tangent cones for rectifiable currents may
or may not exist at a given point. In addition, we make no claim for
uniqueness of such h-cones (the h-cone may depend upon the sequence
of stretches) even when they do exist. However, if $S$ is a mass-minimizing
rectifiable section, then an h-cone will exist over each base point.
Over a regular point, h-cones are simply horizontal planes, but over
singular points they reveal some of the singular structure. 

\begin{thm}
Let $S\in\widetilde{\Gamma}(M)$ be mass-minimizing and continuous
over an open dense subset, as in Theorem {[}\ref{thm:JS}]. At each
point $x_{0}\in\Sigma$, there is an h-cone.
\end{thm}
\begin{proof}
Certainly there is nothing to prove unless $x_{0}$ is a singular
point. In that case, the stretches satisfy, for $\lambda>1$, \begin{eqnarray*}
\mathcal{V}(S_{\lambda,R}) & = & \int_{B(0,R)}\sqrt{1+\frac{1}{\lambda^{2}}\left\Vert \left.\nabla u\right|_{x/\lambda}\right\Vert ^{2}}dA\\
 & \leq & \frac{1}{\lambda}\int_{B(0,R)}\sqrt{1+\left\Vert \left.\nabla u\right|_{x/\lambda}\right\Vert ^{2}}dA\\
 & \leq & \lambda\int_{B(0,R/\lambda)}\sqrt{1+\left\Vert \nabla u\right\Vert ^{2}}rdrd\theta\\
 & = & \lambda f(R/\lambda)\\
 & = & R\frac{f(R/\lambda)}{R/\lambda},\end{eqnarray*}
where $f(t)=\mathcal{V}(S\rest B(0,t)\times S^{1})$, and $S$ is
the graph of $u$ \emph{a.e.} Now, $f$ is increasing, thus is almost-everywhere
differentiable. We have that\begin{eqnarray*}
\frac{d}{dt}\left(\frac{f(t)}{t}\right) & = & \frac{tf'(t)-f(t)}{t^{2}}\\
 & \geq & 0.\end{eqnarray*}
To see this, let $C_{t}$ be the \emph{horizontal cone} over $S\rest\partial B(0,t)\times S^{1}$
defined by extending rays inward horizontally to the fiber over $0$,
that is, if $S=graph(u)$, then $C_{t}$ would be the graph of $v(x):=u(tx/\left\Vert x\right\Vert )$.
Then, since $S$ is volume-minimizing, and $f'(t)$ is the mass of
$S\rest\partial B(0,t)\times S^{1}$ \begin{eqnarray*}
tf'(t) & \geq & \mathcal{V}(C_{t})\\
 & \geq & \mathcal{V}(S\rest B(0,t)\times S^{1})\\
 & = & f(t).\end{eqnarray*}
Thus, $\frac{f(t)}{t}$ is increasing, and so, for $t<R$, $f(t)/t\leq A$,
where $A=f(R)$, or \[
R\left(\frac{f(R/\lambda)}{R/\lambda}\right)\leq RA,\]
and the mass of $S_{\lambda,R}$ is uniformly bounded (in $\lambda$)
for all $\lambda>1$ sufficiently large. Thus any sequence $\left\{ \lambda_{n}\right\} $
of stretches, as $\lambda\to\infty$, is uniformly bounded in mass
over a fixed $R$. In order to apply the compactness theorem, we need
to also show that the mod-2 boundaries $\partial_{2}S_{\lambda_{n}}=S_{\lambda_{n}}\rest\partial B(0,R)\times S^{1}$
have bounded mass. But, by slicing, for any $\lambda>0$, \[
\int_{s/2}^{s}\mathcal{V}\left(\partial\left(S_{\lambda}\rest B(0,r)\right)\right)dr\leq\mathcal{V}\left(S_{\lambda}\rest B(0,s)\right)\leq\lambda f(s/\lambda)\leq sA,\]
following \cite[Theorem 9.8]{Morgan} and \cite[5.4.3(6)]{GMT}, and
so for some $r$, $R/2<r<R$, $\mathcal{V}\left(\partial\left(S_{\lambda}\rest B(0,r)\right)\right)\leq\frac{RA}{R/2}=2A$.
Also, by slicing, almost-all such choices of $r$ have slices that
are rectifiable. Then, the further stretch $S_{(\lambda R/r),R}$
has rectifiable mod-2 boundary $\partial\left(S_{\lambda R/r}\rest B(0,R)\right)$,
with boundary volume $\mathcal{V}\left(\partial\left(S_{\lambda R/r}\rest B(0,R)\right)\right)\leq\left(\frac{R}{r}\right)2A\leq4A$,
So, any sequence $\lambda_{n}\to\infty$ can be modified to one with
a convergent subsequence. Set $S_{0,R}$ to be the limit of this subsequence,
$S_{0,R}:=\lim_{n}S_{\lambda_{n}}\rest B(0,R)\times S^{1}$. Taking
a further subsequence, since the boundaries $\partial_{2}S_{\lambda_{n},R}$
are rectifiable and have no boundary themselves, it can be assumed
that $\partial_{2}S_{\lambda_{n},R}$ also converges, to a rectifiable
section $B\in\widetilde{\Gamma}(S_{R}^{1}\times S^{1})$.

To see that $S_{0}$ is an h-cone, we use the fact that at each point
of $S$ there is an oriented tangent cone \cite[Prop. 4.1]{JS2} in
the usual sense. Any non-vertical ray in the tangent cone at a point
$p\in\pi^{-1}(x_{0})$, under the sequence of horizontal stretches
$\lambda_{n}$, will converge to a horizontal ray in $S_{0,R}$, and
any point of $S_{0,R}$ is on such a horizontal ray, so $S_{0,R}$
is an h-cone.
\end{proof}
Each element of a sequence $S_{\lambda_{i}}$ of horizontal stretches
of a mass-minimizing rectifiable section $S$ in turn minimizes a
modified functional, $\mathcal{V}_{\lambda_{i}}$ defined by\[
\mathcal{V}_{\lambda_{i}}(T):=\mathcal{V}\left(\left(\phi_{\frac{1}{\lambda_{i}}}\right)_{\#}(T)\right)\lambda_{i},\]
where $T\in\widetilde{\Gamma}(B(0,R)\times S^{1})$ with the metric
induced from $M$ by the stretch as before, and $\left(\phi_{\frac{1}{\lambda_{i}}}\right)_{\#}(T)\in\widetilde{\Gamma}(B(0,\frac{R}{\lambda_{i}})\times S^{1})$
has the original metric from $M$. $S_{\lambda_{i}}$ will minimize
$\mathcal{V}_{\lambda_{i}}$ among all rectifiable sections with the
same mod-2 boundary as $\partial_{2}S_{\lambda_{i}}=S_{\lambda_{i}}\rest\partial B(0,R)\times S^{1}$.

The functionals $\mathcal{V}_{\lambda_{i}}$ will converge to a limiting
functional $\mathcal{V}_{0}$. If $T$ is a graph of some smooth function
$u:B(0,R)\to S^{1}$, then $\mathcal{V}(u)=\int_{B(0,R)}\sqrt{1+\left\Vert \nabla u\right\Vert ^{2}}dA$,
and \begin{eqnarray*}
\mathcal{V}_{\lambda_{i}}(T) & := & \int_{B(0,R/\lambda_{i})}\sqrt{1+\lambda_{i}^{2}\left\Vert \left.\nabla u\right|_{\lambda_{i}x}\right\Vert ^{2}}dA\lambda_{i}\\
 & = & \int_{B(0,R)}\sqrt{1+\lambda_{i}^{2}\left\Vert \nabla u\right\Vert ^{2}}\frac{1}{\lambda_{i}^{2}}dA\lambda_{i}\\
 & = & \int_{B(0,R)}\sqrt{\frac{1}{\lambda_{i}^{2}}+\left\Vert \nabla u\right\Vert ^{2}}dA,\end{eqnarray*}
where the second line is just change of variables. On such a current,
clearly\[
\mathcal{V}_{0}(T)=\int_{B(0,R)}\left\Vert \nabla u\right\Vert dA.\]

This functional is called the {}``twisting'' of the current in \cite{BG},
at least in the case of the unit tangent bundle. The main property
of the limiting functional is that it will be minimized by the h-cone
of the volume-minimizer (the minimizer of the limit is the limit of
the minimizers), among all rectifiable sections in $B(0,R)\times S^{1}$
with the same boundary as the h-cone. This property is of course not
the general situation for arbitrary sequences of functionals, but
will hold in this case.

\begin{prop}
If, for some sequence $\lambda_{i}\to\infty$, the stretches $S_{\lambda_{i}}$
converge in $\widetilde{\Gamma}(B(0,R)\times S^{1})$ to $S_{0}$,
where $S$ minimizes $\mathcal{V}$ (and so $S_{\lambda_{i}}$ minimize
$\mathcal{V}_{\lambda_{i}}$) among all such currents with the same
boundary in $\partial B(0,R)\times S^{1}$, then $S_{0}$ minimizes
$\mathcal{V}_{0}$ among all elements of $\widetilde{\Gamma}(B(0,R)\times S^{1})$
with the same boundary as $S_{0}$.
\end{prop}
\begin{proof}
Recall that, taking an appropriate subsequence, the sequence of mod-2
boundaries \[
\partial_{2}S_{\lambda_{i}}=S_{\lambda_{i}}\rest\partial(B(0,R)\times S^{1})\]
converge to $C:=\partial_{2}S_{0}=S_{0}\rest\partial(B(0,R)\times S^{1})$
in $\widetilde{\Gamma}(\partial B(0,R)\times S^{1})$. Then, for any
$\epsilon>0$, there is an $I$ sufficiently large so that, if $i>I$,
then there is a rectifiable current $T_{i}$ in $\partial B(0,R)\times S^{1}$
so that $\partial_{2}T_{i}=\left(\partial_{2}S_{\lambda_{i}}-\partial_{2}S_{0}\right)=\left(S_{\lambda_{i}}-S_{0}\right)\rest\partial(B(0,R)\times S^{1})$
with mass $\mathcal{M}(T_{i})<\epsilon$.

Assume that $S_{0}$ does not minimize $\mathcal{V}_{0}$ among all
rectifiable sections with the same boundary. Then, there is some $\epsilon>0$,
and a rectifiable section $T\in\widetilde{\Gamma}(B(0,R)\times S^{1})$
with $\partial_{2}T=\partial_{2}S_{0}$ so that $\mathcal{V}_{0}(T)<\mathcal{V}_{0}(S_{0})-4\epsilon$.
Choose $I$ above for this $\epsilon$. Since the functionals also
converge, choose $i>I$ sufficiently large so that $\mathcal{V}_{\lambda_{i}}(T)<\mathcal{V}_{0}(T)+\epsilon$,
$\mathcal{V}_{0}(S_{0})<\mathcal{V}_{\lambda_{i}}(S_{0})+\epsilon$,
and $\mathcal{V}_{\lambda_{i}}(S_{0})<\mathcal{V}_{\lambda_{i}}(S_{\lambda_{i}})+\epsilon$
Then, $\partial_{2}\left(T+T_{i}\right)=\partial_{2}S_{\lambda_{i}}$,
$T+T_{i}\in\widetilde{\Gamma}(B(0,R)\times S^{1})$, and\begin{eqnarray*}
\mathcal{V}_{\lambda_{i}}(T+T_{i}) & \leq & \mathcal{V}_{\lambda_{i}}(T)+\mathcal{V}_{\lambda_{i}}(T_{i})\\
 & < & \mathcal{V}_{0}(T)+\epsilon+\mathcal{M}(T_{i})\\
 & < & \mathcal{V}_{0}(T)+2\epsilon\\
 & < & \mathcal{V}_{0}(S_{0})-2\epsilon\\
 & < & \mathcal{V}_{\lambda_{i}}(S_{0})-\epsilon\\
 & < & \mathcal{V}_{\lambda_{i}}(S_{\lambda_{i}}),\end{eqnarray*}
which contradicts the fact that $S_{\lambda_{i}}$ minimizes $\mathcal{V}_{\lambda_{i}}$.
\end{proof}
With this result. we can identify the kind of singular behavior that
can occur.

\subsection{Index of singularities}

A singular point $x\in\Sigma$ of a mass-minimizing $S\in\widetilde{\Gamma}(M)$
as above determines an index, an integer $k_{x}$ generalizing the
index of vector fields.

Let $x\in\Sigma$, and let $S\in\widetilde{\Gamma}(M)$ be mass-minimizing.
For almost-all $\epsilon>0$ sufficiently small, the restriction $S_{\epsilon}:=S\rest\pi^{-1}(\partial B(x,\epsilon))$
determines a rectifiable section $S_{\epsilon}\in\widetilde{\Gamma}(S^{1}\times S^{1})$.
If $P:S^{1}\times S^{1}\to S^{1}$ is the projection onto the second
factor (the fiber), then $k_{\epsilon}$, defined by $P_{\#}(S_{\epsilon})=k_{\epsilon}\left[S^{1}\right]$,
is just the degree of the map. Choose a sequence $\lambda_{i}\to\infty$
so that, on $B(0,1)\times S^{1}$, $S_{\lambda_{i}}$ converges to
an h-cone $H$. Define the \emph{index} of $S$ at $x$, $k_{x}$,
to be \[
k_{x}:=\lim_{i\to\infty}k_{\lambda_{i}}=k_{H},\]
which can be viewed as either a limiting index, or, equivalently,
the index of the h-cone. Of course, the index is also defined for
non-minimizing sections, but does seem to require something like the
existence of an h-cone to guarantee existence and boundedness of the
limit.

\begin{prop}
\label{pro:index-0}If $x$ has index 0, then $x$ is a regular point.
\end{prop}
\begin{proof}
Assume that $x$ is a singular point with index 0. Then, the h-cone
$H$ of $S$ at $x$, in $\widetilde{\Gamma}(B(0,R)\times S^{1})$,
also has index 0, so that $H\rest\partial B(0,R)\times S^{1}$ is
a rectifiable section of degree 0, represented as a map $u:S^{1}\to S^{1}$
of degree 0, possibly with singularities or vertical portions. Thus
$u=e^{if}$ for some real-valued map $f$ (in general, rectifiable
section of the trivial line bundle), and $H$ is the graph of $u(r,\theta)=e^{if(\theta)}$.
The h-cone minimizes the functional \begin{eqnarray*}
\mathcal{V}_{0}(u) & = & \int\int_{B(0,R)}\left\Vert \nabla f\right\Vert rdrd\theta\\
 & = & \int\int_{B(0,R)}\left\Vert \frac{\partial f}{\partial\theta}\right\Vert drd\theta.\end{eqnarray*}
For any function $h(r,\theta)$ with support in the interior of $B(0,1)$,
\begin{eqnarray*}
0 & = & \left.\frac{\partial}{\partial t}\right|_{0}\int\int_{B(0,R)}\left\Vert \nabla(f+th)\right\Vert rdrd\theta\\
 & = & \int\int_{B(0,R)}\frac{\frac{1}{r}\frac{\partial f}{\partial\theta}\frac{\partial h}{\partial\theta}}{\left\Vert \frac{\partial f}{\partial\theta}\right\Vert }rdrd\theta\\
 & = & \int\int_{B(0,R)}\frac{\frac{\partial f}{\partial\theta}\frac{\partial h}{\partial\theta}}{\left\Vert \frac{\partial f}{\partial\theta}\right\Vert }drd\theta.\end{eqnarray*}
For some small $\epsilon>0$, take $h$ to be $h(r,\theta):=\left(M-\epsilon(R-r)/R-f(\theta)\right)_{-}$,
where by $()_{-}$ we mean the nonpositive part of the function, $\left(g\right)_{-}(x):=\inf\left\{ g(x),0\right\} $,
and $M=\sup\left\{ f(\theta)\right\} $. Since, in $Supp(h)$, which
has positive measure, $\partial h/\partial\theta=-\partial f/\partial\theta$,
$\frac{\frac{\partial f}{\partial\theta}\frac{\partial h}{\partial\theta}}{\left\Vert \frac{\partial f}{\partial\theta}\right\Vert }=-\left\Vert \frac{\partial f}{\partial\theta}\right\Vert $.
Unless $f$ is \emph{a.e.} constant, the integral will be negative,
which would contradict minimality of $S$. Thus $f$ must be constant,
and so the graph of $S$ is continuous at $x$.
\end{proof}
As a corollary, we can now show that at any singular point, the entire
fiber over the point is contained in the support of $S$.

\begin{cor}
If $x\in\Sigma$ is a singular point for a mass-minimizing $S\in\widetilde{\Gamma}(M)$,
then $\pi^{-1}(x)\subset Supp(S)$.
\end{cor}
\begin{proof}
The singularity has to have nonzero index, which implies that the
entire fiber is in the support.
\end{proof}
\begin{cor}
The singular points of a mass-minimizing $S\in\widetilde{\Gamma}(M)$
are isolated. 
\end{cor}
\begin{proof}
If a singular point $x_{0}$ is the limit of other singular points
and is of index $k$, consider the current defined in $B(0,R)\times S^{1}$
as $Sv$, where $v(r,\theta)=e^{-ik\theta}$, which has support \[
Supp(Sv)=\left\{ \left.(x,yv(x))\in B(0,R)\times S^{1}\right|(x,y)\in S\right\} \]
 and has the obvious tangent planes and multiplicities inherited from
$S$. Clearly $Sv$ has index 0 at $x_{0}$. However, since there
is a sequence of singular points of $S$ approaching $x_{0}$, and
for each such point $x_{i}$ the entire fiber is contained in the
support of $S$, this will also be true for $Sv$ since $v$ has only
$x_{0}$ as a singular point, so is regular at each $x_{i}$, and
thus $Sv$ is still singular at $x_{i}$ as claimed. Since $Supp(Sv)$
is closed, it must contain the entire fiber over $x_{0}$, so $x_{0}$
is a singular point of $Sv$, and is of index 0 at $x_{0}$. 

$Sv$ does not minimize the volume, but it does minimize a twisted
volume $\mathcal{V}_{v}$ defined for rectifiable sections of $B(0,R)\times S^{1}$
by $\mathcal{V}_{v}(T)=\mathcal{V}(Tv^{-1})$. Stretching as before,
the h-cone $Hv$ (where $H$ is the h-cone of $S$ for a specific
sequence of stretches) will minimize \[
\mathcal{V}_{v,0}(w)=\int\int_{B(0,R)\times S^{1}}\left\Vert \nabla(wv^{-1})\right\Vert dA.\]
Since $Hv$, the minimizer of this functional (with the boundary conditions
inherited from $S$), is an h-cone, we have the variational condition,
as before, if $w=e^{if}$, \begin{eqnarray*}
0 & = & \left.\frac{d}{dt}\right|_{0}\int_{0}^{R}\int_{0}^{2\pi}\left\Vert \nabla(f+th(r,\theta)+k\theta)\right\Vert rd\theta dr\\
 & = & \int_{0}^{R}\int_{0}^{2\pi}\frac{\left(\frac{\partial f}{\partial\theta}+k\right)\frac{\partial h}{\partial\theta}}{\left\Vert \frac{\partial f}{\partial\theta}+k\right\Vert }d\theta dr.\end{eqnarray*}
As with Proposition (\ref{pro:index-0}), taking $h$ to be $h(r,\theta):=\left(M-\epsilon(R-r)/R-f(\theta)-k\theta\right)_{-}$,
where $M$ is the maximum of $f$, would provide a contradiction unless
$\frac{\partial f}{\partial\theta}+k\equiv0$.

Thus, the h-cone of $S$ at $x_{0}$ is that of $v^{-1}$ itself,
which has only the singularity at 0. If $S$ had a sequence of singularities
approaching $x_{0}$, the h-cone would also. Thus the singularities
of $S$ are isolated.
\end{proof}
This result also shows:

\begin{cor}
Each singularity is of index $\pm2$.
\end{cor}
\begin{proof}
In \cite{BG}, it is shown that only an isolated singularity of index
$\pm2$ can be a rectifiable section.
\end{proof}
We finally are in a position to prove the main result, which is

\begin{thm}
Let $M$ be a circle bundle over a compact Riemann surface $\Sigma$,
with the Sasaki metric, so that the Euler number $e(M)$ of the circle
bundle is even. Then, there is a mass-minimizing rectifiable section
$S$ which is moreover a smooth, embedded minimal surface in $M$.
Topologically, $S$ is $\Sigma$ with a finite number of cross-caps.
\end{thm}
\begin{proof}
For a volume-minimizing section $S\in\widetilde{\Gamma}(M)$ as shown
to exist by Theorem {[}\ref{thm:JS}], we have shown that there are
a finite number of singular points $x\in\Sigma$, over each of which
$S$ contains the entire fiber, and the index is $\pm2$.  On the
complement of those singular fibers, $S$ is a continuous graph and
is a minimal surface, so it is smooth (since it is codimension 1 in
a 3-manifold.). 

In an $\epsilon$-neighborhood of the singular fibers, the graph is
asymptotically that of $u=e^{\pm2i\theta}$, and so the current is
$C^{1}$ at these points. Since it is of class $C^{1}$ and minimal
(weak mean curvature vanishing), it is a smooth minimal surface. That
the structure of the surface in a neighborhood of a singularity is
a cross-cap can be found in \cite{BG}. The topological statement
then follows.
\end{proof}
\begin{rem}
While it seems clear that a volume-minimizing rectifiable section
should have a minimal number of singular points, since singularities
add to volume \cite{BCJ}, we do not make that claim. If, however,
the Euler class of the bundle is 0, there will be a smooth minimizer
which is a global section, since in that case we can work with the
class of currents which are limits of smooth sections. All singularities
would then be of index 0, and so could not exist.
\end{rem}

\begin{rem}
A similar statement should hold for bundles with odd Euler class,
except that the minimizer will not be a section, but will generically
be the double of a section with a finite number of singular fibers.
The Hopf fibration $S^{3}\to S^{2}$ provides such an example, with
the equatorial $S^{2}$ in $S^{3}$ being the double-section, meeting
each fiber at two points except for one fiber contained in the surface.
\end{rem}


\begin{thebibliography}{1}
\bibitem{BG}V. Borrelli and O. Gil Medrano, \emph{Area-minimizing
vector fields on round 2-spheres, }to appear.

\bibitem{BCJ} F. G. B. Brito, P. M. Chac\'{o}n, and D. L. Johnson,\emph{
Unit fields on punctured spheres}, Bulletin de la Société Mathématique
de France,\textbf{ 136 }(1) (2008), 147-157.

\bibitem{GMT}H. Federer, \textit{Geometric Measure Theory,} Springer-Verlag
1969. 

\bibitem{JS2} D. L. Johnson and P. Smith,\emph{ Partial regularity
of mass-minimizing rectifiable sections}, Annals of Global Analysis
and Geometry \textbf{30} (2006), 239-287. 

\bibitem{Morgan}F. Morgan, \emph{Geometric Measure Theory, A Beginner's
Guide}, Academic Press, fourth edition, 2008.

\bibitem{Sasaki}T. Sasaki, \textit{On the differential geometry of
tangent bundles of Riemannian manifolds}, T\^ ohoku Math. J. \textbf{10}
(1958), 338--354.
\end{thebibliography}
\end{document}